\documentclass[12pt]{amsart}
\usepackage{geometry}
\usepackage{array}
\usepackage{amssymb,amsthm,amsmath}
\usepackage{ulem}
\usepackage{moreverb}
\usepackage[all]{xy}
\usepackage{graphicx}

\usepackage{mathtools}
\usepackage{enumitem}
\usepackage{hyperref}
\usepackage[backend=bibtex,citestyle=alphabetic,bibstyle=alphabetic,maxbibnames=99]{biblatex}
\renewbibmacro*{doi+eprint+url}{%
  \iftoggle{bbx:doi}
    {\iffieldundef{url}{\printfield{doi}}{}}
    {}%
  \newunit\newblock
  \iftoggle{bbx:eprint}
    {\usebibmacro{eprint}}
    {}%
  \newunit\newblock
  \iftoggle{bbx:url}
    {\usebibmacro{url+urldate}}
    {}}

\makeatletter
\renewcommand{\p@subsection}{\thesection.}
\makeatother

\providecommand{\phantomsection}{}
\AtBeginDocument{\let\textlabel\label}
\makeatletter
\newcommand{\mylabel}[2]{\raisebox{.7\normalbaselineskip}{\phantomsection}#1%
  \def\@currentlabel{#1}\textlabel{#2}}
\newcommand{\mylabelrev}[2]{\raisebox{.7\normalbaselineskip}{\phantomsection}%
  \def\@currentlabel{#1}\textlabel{#2}}
\makeatother

\newcounter{eqsign}

\newtheorem{thm}{Theorem}
\newtheorem{prop}[thm]{Proposition}

{\theoremstyle{definition}  }
{}
{\theoremstyle{remark} }

\newtheorem{bigclm}[thm]{Claim}
{\theoremstyle{remark} }

 \newcommand{\RR}{\mathbb{R}}

\newcommand{\NN}{\mathbb{N}}

\newcommand{\disp}[0]{\displaystyle}
\DeclarePairedDelimiter{\bracepair}{\lbrace}{\rbrace}

\DeclarePairedDelimiter{\parenpair}{(}{)}

\newcommand{\declare}{:=}

\DeclareMathOperator{\mes}{mes}
\newcommand{\dmes}[1]{\mes_{#1}}
\DeclareMathOperator{\grad}{grad}
\newcommand{\analy}[1]{\mathcal{A}(#1)}
\DeclareMathOperator{\dist}{dist}
\newcommand{\disty}[2]{\dist \parenpair*{#1, #2}}
\newcommand{\bigder}{\mathcal{D}}

\begin{document}
\title[Zero-Sets of Real Analytic Functions]{The zero set of a real analytic function}
\author{Boris S. Mityagin}
\address{231 West 18th Avenue, The Ohio State University, Columbus, OH 43210}
\email{\url{mityagin.1@osu.edu}}

\begin{abstract}
A brief proof of the statement that the zero-set of a nontrivial real-analytic function in $d$-dimensional space has zero measure is provided.
\end{abstract}
\maketitle
\normalem

Recently, in postings \cite{Dang,Kuchment} the authors face the problem of finding a good reference to the proof (or even an explicit statement) of the following.

\begin{prop} \label{prop:zerosetzeromeas}
Let $A(x)$ be a real analytic function on (a connected open domain $U$ of) $\RR^d$.  If $A$ is not identically zero, then its zero set 
\begin{equation} \label{eq:zerosetdef}
F(A) \declare \bracepair{x \in U: A(x) = 0}
\end{equation}
has a zero measure, i.e., $\dmes{d} F(A) = 0$.
\end{prop}

Remark 5.23 in \cite{Kuchment} and Lemma 1.2 in \cite{Dang} list possible approaches to the proof --- from using Fubini's theorem to Hironaka's resolution of singularities.  This posting suggests a proof on the level of Calculus II.  

For any $B \in \analy{U}$, the space of real-valued analytic functions on $U$, let us introduce the set
\begin{equation} \label{eq:zero}
G(B) \declare \bracepair{x \in F(B): \vert (\grad B)(x) \vert \neq 0}.
\end{equation}

\begin{bigclm} \label{clm:zerosetnonzeroderzeromeas} If $B \in \analy{U}$, then 
\begin{equation} \label{eq:measnonzerograd}
\dmes{d} G(B) = 0
\end{equation}
\end{bigclm}
\begin{proof}
Indeed, $G(B)$ is a union of the sets
\begin{equation*} 
\begin{split}
G_k &\declare \left\lbrace x \in F(B): \vert x \vert \leq k, \right.\\
& \qquad \left. \vert (\grad B)(x) \vert \geq \disp \frac{1}{k} \, , \quad \disty{x}{U^{\complement}} \geq \frac{1}{k} \right\rbrace .
\end{split}
\end{equation*}
[The third condition is omitted if $U \equiv \RR^d$, i.e., $U^{\complement} = \varnothing$.]  If $B \equiv 0$, all sets $G_k$ are empty.  Otherwise, for these compact sets, by the Implicit Function Theorem (e.g., \cite[pp. 195--201, Thm. 9.18]{Rudin}, $n = 1$) any point $y \in G_k$ has a ball/neighborhood where a piece of $F(A)$ is in a coordinate $(d-1)$--dimensional subspace of a nice coordinate system.  Choose a proper finite covering and get
\begin{equation}
\dmes{d} G_k = 0, \quad \forall k.
\end{equation}
So \eqref{eq:measnonzerograd} follows.  
\end{proof}

\begin{bigclm} \label{clm:faembed}
Assume
\begin{equation} \label{eq:inithyp} 
A \in \analy{U} \text{ is not identically zero}.
\end{equation}
Then $F(A)$ is contained in a countable union of sets of type $G(B)$.  
\end{bigclm}
\begin{proof}
Indeed, for any $w \in F(A)$ there exists an integer $n \geq 0$ such that
\begin{subequations} \label{eq:realanaly}
\begin{align}
(\bigder^t A)(w) = 0, \quad \forall t \in \NN^d, \quad \NN = \bracepair{0, 1, 2, \dotsc}, \quad \text{if }\vert t \vert \leq n \label{eq:manyzero}\\
\intertext{but}
(\bigder^s A)(w) \neq 0 \text{ for some } s \in \NN^d, \quad \vert s \vert = n + 1. \label{eq:notallzero}
\end{align}
\end{subequations}
Otherwise, all derivatives $(\bigder^p A)(w)$, $p \in \NN^d$, would vanish at $w$ in contradiction to \eqref{eq:inithyp}.  As usual, we put $\disp 
\bigder^t \declare \frac{\partial^{\vert t \vert}}{\partial x_1^{t_1}\partial x_2^{t_2} \dotsc \partial x_{d}^{t_d}}$, 
and $\vert a \vert = \vert a_1 \vert + \dotsb + \vert a_d \vert$ for any $a \in \RR^d$.  We also define $e_j$ to be the $j$th standard basis vector, i.e.,
\[
e_j \declare ( 0, 0, \dotsc, 0, \overbrace{1}^{j - \text{th slot}}, 0, \dotsc, 0) \in \RR^d.
\]

Notice that
\begin{equation}
\left( \grad (\bigder^p B) \right)(x) = \bracepair{(\bigder^{p + e_j} B)(x), \, \, 1 \leq j \leq d}
\end{equation}
for any $p \in \NN^d$.  Therefore, by \eqref{eq:manyzero}, \eqref{eq:notallzero}
\begin{equation*}
w \in G(\bigder^q A) \quad \text{ for some } \quad q, \quad \vert q \vert = n;
\end{equation*}
it proves that
\begin{equation} \label{eq:fainclude}
F(A) \subseteq \bigcup_{t \in \NN^d} G(\bigder^t A).
\end{equation}
\end{proof}
By Claim~\ref{clm:zerosetnonzeroderzeromeas} all sets on the right side of \eqref{eq:fainclude} have $\dmes{d}$ -- zero and by Claim~\ref{clm:faembed}
\[
\dmes{d} F(A) = 0.
\]
Proposition~\ref{prop:zerosetzeromeas} is proved.  

In conclusion, let us make a few remarks.
\begin{enumerate}[label=(\roman*)]
\item Curiously enough, the above proof does not use Lebesgue measure theory (Fubini's Theorem or whatever else).  We talk only about sets of measure zero; by definition,
\begin{equation*}
\dmes{d} C = 0 \quad \text{ for } \quad C \subseteq \RR^d
\end{equation*}
if for any $\epsilon > 0$ there exists a system $\disp \bracepair*{B = B(y; \delta_y)}_{y \in Y}$ of balls in $\RR^d$ such that
\begin{equation*} 
\sum_y \delta(y)^d < \epsilon \quad \text{ and } \quad C \subseteq \bigcup_{y \in Y} B(y, \delta(y)).
\end{equation*}
\item Along the same lines we can explain (prove)
\begin{prop}
Under the assumptions of Proposition~\ref{prop:zerosetzeromeas}, the Hausdorff dimension of $F(A)$ does not exceed $(d-1)$.  
\end{prop}
\item The statements in Proposition~\ref{prop:zerosetzeromeas} and Claims~\ref{clm:zerosetnonzeroderzeromeas} and \ref{clm:faembed}, in particular the inclusion \eqref{eq:fainclude}, remain valid if we assume that $A$ is a real function of some quasi-analtyic class, for example, 
\begin{equation*}
\begin{split}
A \in C_M(U) \declare & \left\lbrace f \in C^{\infty}(U): \vert (\bigder^t f)(x) \vert \leq C_K R_K^{\vert t \vert} M(\vert t \vert), \quad x \in K, \right. \\
& \left. \qquad \forall \text{ compact } K \subset U \right\rbrace,
\end{split}
\end{equation*}
where $M = \bracepair{M(n): n \in \NN}$, $M(0) = 1$, $M(n) \nearrow \infty$, $\log M(n)$ is convex, and 
\[
\sum_{n = 1}^{\infty} \frac{M(n)}{M(n + 1)} = \infty.  
\]
\end{enumerate}

\nocite{*}
\printbibliography

\end{document}